\documentclass{article}

\usepackage{amsfonts}
\usepackage{amsthm}
\usepackage{amsmath}
\usepackage{amssymb}
\usepackage{mathtools} %for prescript
\usepackage{authblk} %for affil
\usepackage{xcolor}
\usepackage{soul} %for strikethrough

\usepackage[top=1in, bottom=1in, left=1in, right=1in]{geometry}

\newtheorem{thm}{Theorem}[section]
\newtheorem{lem}{Lemma}[section]
\newtheorem{coroll}{Corollary}[section]

\theoremstyle{definition}
\newtheorem{defn}{Definition}[section]

\theoremstyle{remark}
\newtheorem{remark}{Remark}[section]
\newtheorem{example}{Example}[section]

\DeclareMathOperator*{\esssup}{ess\,sup}

\definecolor{cucol}{rgb}{0,0,0.8}
\definecolor{afcol}{rgb}{1,0,0}

\begin{document}

%\begin{frontmatter}

\title{Analytical Development of Incomplete Riemann--Liouville Fractional Calculus}

%%this line removes the date, but space is still left for it;
%if used, remove the \vspace{-1cm}
\date{}

%this gives the date in the form Mon 30 Jan 2012, 8:57pm;
%if used, retain the \vspace{-1cm}
%\date{\shortdayofweekname{\day}{\month}{\year}{ }\mydate\today}

\author[1]{Arran Fernandez\thanks{Corresponding author. Email: \texttt{arran.fernandez@emu.edu.tr}}}
\author[2]{Ceren Ustao\u{g}lu}
\author[1]{Mehmet Ali \"Ozarslan}

\affil[1]{{\small Department of Mathematics, Faculty of Arts and Sciences, Eastern Mediterranean University, Gazimagusa, TRNC, Mersin 10, Turkey}}
\affil[2]{{\small Department of Computer Engineering, Faculty of Engineering, Final International University, Kyrenia, TRNC, Mersin 10, Turkey}}

% Latex won't make the title unless told:

\maketitle

%%to remove the space left for date, use:
%\vspace{-1cm}

\begin{abstract}
\noindent The theory of fractional calculus has developed in a number of directions over the years, including:
\begin{itemize}
\item the formulation of multiple different definitions of fractional differintegration,
\item the extension of various properties of standard calculus into the fractional scenario,
\item the application of fractional differintegrals to assorted special functions.
\end{itemize}

Recently, a new variant of fractional calculus has arisen, namely incomplete fractional calculus. In two very recent papers, incomplete versions of the Riemann--Liouville and Caputo fractional differintegrals have been formulated and applied to several important special functions. However, this recent development in the field still requires further analysis. 

In the current work, we develop the theory of incomplete fractional calculus in more depth than has been done before, investigating the further properties of the incomplete Riemann-Liouville fractional differintegrals and answering some fundamental questions about how these operators work.

By considering appropriate function spaces, we formulate rigorously the definitions of incomplete Riemann-Liouville fractional integration, and justify how this model may be used to analyse a wider class of functions than classical fractional calculus can consider. By using the idea of analytic continuation from complex analysis, we formulate definitions for incomplete Riemann-Liouville fractional differentiation, hence extending the incomplete integrals to a fully-fledged model of fractional calculus.

We also investigate and analyse these operators further, in order to prove new properties of the incomplete Riemann--Liouville fractional calculus. These include a Leibniz rule for incomplete differintegrals of products, and composition properties of incomplete differintegrals with classical calculus operations. These are natural and expected issues to investigate in any new model of fractional calculus, and in the incomplete Riemann--Liouville model the results emerge naturally from the definition previously proposed.
\end{abstract}

\section{Introduction} \label{sec:intro}

The field of fractional calculus has its roots in the question, posed by L'Hopital to Leibniz in the 17th century, of what would happen to the operation of multiple differentiation $\frac{\mathrm{d}^ny}{\mathrm{d}x^n}$ if the order $n$ were taken to be $\frac{1}{2}$. During the 18th and 19th centuries, this question, and also the broader issue of extending $n$ to any real or complex value, was answered in a number of different ways. Thus, several competing definitions were created for fractional differentiation and integration (often referred to together as fractional \textit{differintegration}). These included what are now referred to as the Riemann--Liouville and Gr\"unwald--Letnikov models of fractional calculus. For a more detailed discussion of the history of fractional calculus up to the late 20th century, we refer the reader to \cite{dugowson,miller-ross}.

In more recent decades, interest in the field has been increasing rapidly. Partly this is due to the discovery of practical applications in various areas including fluid dynamics, chaos theory, bioengineering, etc. \cite{hilfer,hristov,magin,mainardi,west}. Partly also the expansion is due to the realisation that the classical definitions of fractional differintegrals are only the tip of the iceberg: dozens of other models can be proposed and analysed, each with their own properties and applications \cite{atangana-baleanu,prabhakar,kilbas-saigo-saxena,ozarslan-ozergin,cetinkaya-kiymaz-agarwal-agarwal,jarad-abdeljawad-alzabut}.

The most frequently used model of fractional calculus is the \textbf{Riemann--Liouville} one, in which fractional integrals are defined using a power-function kernel and fractional derivatives are defined using standard derivatives of fractional integrals:
\begin{alignat}{2}
\label{RLdef:int} \prescript{RL}{a}I^{-\mu}_xf(x)&=\prescript{RL}{a}D^{\mu}_xf(x)\coloneqq\frac{1}{\Gamma(-\mu)}\int_a^x(x-t)^{-\mu-1}f(t)\,\mathrm{d}t,\quad\quad&\mathrm{Re}(\mu)<0; \\
\label{RLdef:deriv} \prescript{RL}{a}D^{\mu}_xf(x)&\coloneqq\frac{\mathrm{d}^n}{\mathrm{d}x^n}\prescript{RL}{a}I^{n-\mu}_xf(x),\quad n\coloneqq\lfloor\mathrm{Re}(\mu)\rfloor+1,\quad\quad&\mathrm{Re}(\mu)\geq0.
\end{alignat}
The definition \eqref{RLdef:int} of Riemann--Liouville (RL) fractional integrals is valid for $x\in(a,b)$ and $f\in L^1(a,b)$, but these are not necessary conditions: we can if desired replace the $L^1$ space by other function spaces such as the space of absolutely continuous functions \cite{miller-ross,samko-kilbas-marichev}. The definition \eqref{RLdef:deriv} of RL fractional derivatives is valid for $x\in(a,b)$ and $f\in C^n(a,b)$, although again these are not the only viable set of conditions to impose \cite{miller-ross,samko-kilbas-marichev}.

Many of the newer alternative models of fractional calculus, some of which we shall discuss in detail below, involve changing the kernel in \eqref{RLdef:int} from a power function to some more complex special function. This is useful because, by using many different kernel functions with different behaviours, we are able to model a wider spectrum of different fractional systems which all behave in different ways.

Some special functions which have a particularly strong connection with fractional calculus are the \textbf{incomplete gamma} and \textbf{incomplete beta} functions, defined as follows. The upper and lower incomplete gamma functions are respectively
\begin{alignat}{2}
\label{IGFdef:upper} \Gamma(\nu,x)&\coloneqq\int_x^{\infty}t^{\nu-1}e^{-t}\,\mathrm{d}t,\quad\quad&\mathrm{Re}(\nu)>0; \\ 
\label{IGFdef:lower} \gamma(\nu,x)&\coloneqq\int_0^xt^{\nu-1}e^{-t}\,\mathrm{d}t,\quad\quad&\mathrm{Re}(\nu)>0.
\end{alignat}
The incomplete beta function is
\begin{equation}
\label{IBFdef} B_y(a,b)\coloneqq\int_0^yt^{a-1}(1-t)^{b-1}\,\mathrm{d}t,\quad\quad0\leq y\leq1,\mathrm{Re}(a)>0,\mathrm{Re}(b)>0.
\end{equation}
To see the significance of these functions in fractional calculus, let us consider the Riemann--Liouville differintegrals of some of the most fundamental elementary functions: namely, exponential functions and power functions. The following results are proved in \cite{miller-ross}:
\begin{alignat}{2}
\label{RLex:exp} \prescript{RL}{a}D^{\mu}_x(e^{\alpha x})&=\frac{\alpha^{\mu}e^{\alpha x}}{\Gamma(-\mu)}\gamma(-\mu,\alpha(x-a)),\quad\quad&&\mu\in\mathbb{C},\alpha\neq0; \\
\label{RLex:power} \prescript{RL}{a}D^{\mu}_x(x^{\alpha})&=\frac{x^{\alpha-\mu}}{\Gamma(-\mu)}B_{\frac{x-a}{x}}(-\mu,\alpha+1),\quad\quad&&\mathrm{Re}(\mu)<0,\mathrm{Re}(\alpha)>-1.
\end{alignat}
When fractional differintegral operators are applied, some of the most basic functions of calculus become relatives of the incomplete gamma and beta functions. Thus, these incomplete functions are in fact fundamental to the field of fractional calculus, and it is worth studying them in more detail to understand the connection between fractionality and incompleteness.

Recently, a new type of fractional calculus was defined which is called \textbf{incomplete Riemann--Liouville} fractional calculus \cite{ozarslan-ustaoglu1}. The underlying idea is to consider the same operation of ``incompletifying'' that leads us from the integrals defining the gamma and beta functions to those defining the incomplete gamma and beta functions, and apply this same operation to the integral \eqref{RLdef:int} defining the Riemann--Liouville fractional integral. This gives rise to the following equivalent expressions for the lower incomplete Riemann--Liouville fractional integral:
\begin{align}
\label{IRLdef:intlower1} \prescript{RL}{0}D_x^{\mu}[f(x);y]&=\frac{1}{\Gamma(-\mu)}\int_0^{yx}(x-t)^{-\mu-1}f(t)\,\mathrm{d}t \\
\label{IRLdef:intlower2} &=\frac{x^{-\mu}}{\Gamma(-\mu)}\int_0^y(1-u)^{-\mu-1}f(ux)\,\mathrm{d}u \\
\label{IRLdef:intlower3} &=\frac{x^{-\mu}y}{\Gamma(-\mu)}\int_0^1(1-wy)^{-\mu-1}f(ywx)\,\mathrm{d}w,\quad\quad\mathrm{Re}(\mu)<0.
\end{align}
And for the upper incomplete Riemann--Liouville fractional integral:
\begin{align}
\label{IRLdef:intupper1} \prescript{RL}{0}D_x^{\mu}\{f(x);y\}&=\frac{1}{\Gamma(-\mu)}\int_{yx}^x(x-t)^{-\mu-1}f(t)\,\mathrm{d}t \\
\label{IRLdef:intupper2} &=\frac{x^{-\mu}}{\Gamma(-\mu)}\int_y^1(1-u)^{-\mu-1}f(ux)\,\mathrm{d}u \\
\label{IRLdef:intupper3} &=\frac{x^{-\mu}y}{\Gamma(-\mu)}\int_0^{1-y}v^{-\mu-1}f((1-v)x)\,\mathrm{d}v,\quad\quad\mathrm{Re}(\mu)<0.
\end{align}
In the seminal work \cite{ozarslan-ustaoglu1}, the incomplete Riemann--Liouville fractional integral operators were applied to some elementary and special functions, as example results to establish their validity. This paper was followed by another \cite{ozarslan-ustaoglu2} in which variants of Caputo type were defined for these operators. Thus, the field of incomplete fractional calculus has been opened for investigation. There is still much to be done in this field, ranging from fundamental properties such as the function spaces on which the operators can be defined, to more advanced results such as Leibniz's rule. In the current work, we aim to investigate and establish a number of results concerning the already defined incomplete Riemann--Liouville fractional integrals, and also to introduce and analyse some related operators of incomplete fractional type.

%{\color{red}
%\begin{itemize}
%\item introduce fractional calculus
%\item importance of incomplete gamma/beta functions in FC
%\item introduce incomplete FC and its analysis so far
%\item justify the need for further analysis and more models of incomplete FC
%\end{itemize}
%}

\section{A rigorous analysis of incomplete Riemann--Liouville fractional calculus} \label{sec:rigour}

\subsection{Function spaces for the fractional integrals}

It is known that the standard Riemann--Liouville fractional integral \eqref{RLdef:int} is defined for $x\in[a,b]$ and $f\in L^1[a,b]$. For the incomplete Riemann--Liouville fractional integrals \eqref{IRLdef:intlower1}--\eqref{IRLdef:intupper3}, we have taken the lower bound to be $a=0$, so it can be assumed that $x$ lies in a fixed interval $[0,b]$. In order to formulate a fully rigorous definition, we also need to consider the conditions on the function $f$, and specify a function space for $f$ such that the incomplete RL fractional integrals of $f$ are well-defined.

\begin{thm}
\label{IRL:L1lower}
If $b>0$ and $0<y<1$ and $\mu\in\mathbb{C}$ with $\mathrm{Re}(\mu)>0$, then the $\mu$th lower incomplete Riemann--Liouville fractional integral defines a bounded operator \[\prescript{RL}{0}D^{-\mu}[\cdot;y]:L^1[0,yb]\rightarrow L^1[0,b].\]
\end{thm}

\begin{proof}
Let $f$ be a function defined on $[0,b]$. We need to prove that the $L^1[0,b]$ norm of the function $\prescript{RL}{0}D_x^{-\mu}[f(x);y]$ is uniformly bounded in terms of the $L^1[0,yb]$ norm of $f$. Note that here we are defining $\mu$ to be the order of integration, not the order of differentiation, so its sign is reversed from the earlier expressions.

We start from the definition \eqref{IRLdef:intlower1}. For any $x\in[0,b]$,
\begin{align*}
\Big|\prescript{RL}{0}D_x^{-\mu}[f(x);y]\Big|&\leq\frac{1}{|\Gamma(\mu)|}\int_0^{yx}|f(t)|(x-t)^{\mathrm{Re}(\mu)-1}\,\mathrm{d}t \\
&\leq\frac{1}{|\Gamma(\mu)|}\left(\sup_{[0,yx]}(x-t)^{\mathrm{Re}(\mu)-1}\right)\int_0^{yx}|f(t)|\,\mathrm{d}t.
\end{align*}
The value of this supremum depends on the sign of $\mathrm{Re}(\mu)-1$. Thus, there are two cases to be considered according to the value of $\mu$.

\textbf{Case 1: $\boldsymbol{0<\mathrm{Re}(\mu)\leq1}$.} Here the supremum occurs at $t=yx$, so we have
\begin{align*}
\Big|\prescript{RL}{0}D_x^{-\mu}[f(x);y]\Big|&\leq\frac{(x-yx)^{\mathrm{Re}(\mu)-1}}{|\Gamma(\mu)|}\int_0^{yx}|f(t)|\,\mathrm{d}t \\
&\leq\frac{(x-yx)^{\mathrm{Re}(\mu)-1}}{|\Gamma(\mu)|}\Big\|f(t)\Big\|_{L^1[0,yb]}.
\end{align*}
Integrating this inequality over all $x\in[0,b]$, we deduce that
\begin{align}
\nonumber \Big\|\prescript{RL}{0}D^{-\mu}[f;y]\Big\|_{L^1[0,b]}&\leq\int_0^b\frac{(x-yx)^{\mathrm{Re}(\mu)-1}}{|\Gamma(\mu)|}\Big\|f(t)\Big\|_{L^1[0,yb]}\,\mathrm{d}x \\
\label{IRL:L1lower:eqn1} &=\frac{(1-y)^{\mathrm{Re}(\mu)-1}b^{\mathrm{Re}(\mu)}}{|\Gamma(\mu)|\mathrm{Re}(\mu)}\Big\|f(t)\Big\|_{L^1[0,yb]}.
\end{align}
The fraction coefficient on the right-hand side depends only on $b$, $y$, and $\mu$, so we have a bound of the desired form in this case.

\textbf{Case 2: $\boldsymbol{\mathrm{Re}(\mu)>1}$.} Here the supremum over $t\in[0,yx]$ of the function $(x-t)^{\mathrm{Re}(\mu)-1}$ occurs at $t=0$, so we have
\begin{align*}
\Big|\prescript{RL}{0}D_x^{-\mu}[f(x);y]\Big|&\leq\frac{x^{\mathrm{Re}(\mu)-1}}{|\Gamma(\mu)|}\int_0^{yx}|f(t)|\,\mathrm{d}t \\
&\leq\frac{x^{\mathrm{Re}(\mu)-1}}{|\Gamma(\mu)|}\Big\|f(t)\Big\|_{L^1[0,yb]}.
\end{align*}
Integrating this inequality over all $x\in[0,b]$, we deduce that
\begin{align}
\nonumber \Big\|\prescript{RL}{0}D^{-\mu}[f;y]\Big\|_{L^1[0,b]}&\leq\int_0^b\frac{x^{\mathrm{Re}(\mu)-1}}{|\Gamma(\mu)|}\Big\|f(t)\Big\|_{L^1[0,yb]}\,\mathrm{d}x \\
\label{IRL:L1lower:eqn2} &=\frac{b^{\mathrm{Re}(\mu)}}{|\Gamma(\mu)|\mathrm{Re}(\mu)}\Big\|f(t)\Big\|_{L^1[0,yb]}.
\end{align}
Again, the fraction on the right-hand side depends only on $b$, $y$, and $\mu$, so we have a bound of the desired form.
\end{proof}

\begin{thm}
\label{IRL:L1upper}
If $b>0$ and $0<y<1$ and $\mu\in\mathbb{C}$ with $\mathrm{Re}(\mu)>1$, then the $\mu$th upper incomplete Riemann--Liouville fractional integral defines a bounded operator \[\prescript{RL}{0}D^{-\mu}\{\cdot;y\}:L^1[0,b]\rightarrow L^1[0,b].\]
\end{thm}

\begin{proof}
Let $f$ be a function defined on $[0,b]$. We need to prove that the $L^1[0,b]$ norm of the function $\prescript{RL}{0}D_x^{-\mu}\{f(x);y\}$ is uniformly bounded in terms of the $L^1[0,b]$ norm of $f$. Again $\mu$ is the order of integration, not the order of differentiation, so its sign is reversed from the earlier expressions \eqref{IRLdef:intupper1}--\eqref{IRLdef:intupper3}.

We start from the definition \eqref{IRLdef:intupper1}. For any $x\in[0,b]$,
\begin{align*}
\Big|\prescript{RL}{0}D_x^{-\mu}\{f(x);y\}\Big|&\leq\frac{1}{|\Gamma(\mu)|}\int_{yx}^x|f(t)|(x-t)^{\mathrm{Re}(\mu)-1}\,\mathrm{d}t \\
&\leq\frac{1}{|\Gamma(\mu)|}\left(\sup_{[yx,x]}(x-t)^{\mathrm{Re}(\mu)-1}\right)\int_{yx}^x|f(t)|\,\mathrm{d}t.
\end{align*}
Since we assumed $\mathrm{Re}(\mu)>1$, the supremum occurs at $t=yx$. (In this case, if we had $0<\mathrm{Re}(\mu)<1$, the supremum would be infinite due to the blowup at $t=x$.) So we have
\begin{align*}
\Big|\prescript{RL}{0}D_x^{-\mu}\{f(x);y\}\Big|&\leq\frac{(x-yx)^{\mathrm{Re}(\mu)-1}}{|\Gamma(\mu)|}\int_{yx}^x|f(t)|\,\mathrm{d}t \\
&\leq\frac{(x-yx)^{\mathrm{Re}(\mu)-1}}{|\Gamma(\mu)|}\Big\|f(t)\Big\|_{L^1[0,b]}.
\end{align*}
Integrating this inequality over all $x\in[0,b]$, we deduce that
\begin{align}
\nonumber \Big\|\prescript{RL}{0}D^{-\mu}\{f;y\}\Big\|_{L^1[0,b]}&\leq\int_0^b\frac{(x-yx)^{\mathrm{Re}(\mu)-1}}{|\Gamma(\mu)|}\Big\|f(t)\Big\|_{L^1[0,b]}\,\mathrm{d}x \\
\label{IRL:L1upper:eqn} &=\frac{(1-y)^{\mathrm{Re}(\mu)-1}b^{\mathrm{Re}(\mu)}}{|\Gamma(\mu)|\mathrm{Re}(\mu)}\Big\|f(t)\Big\|_{L^1[0,yb]}.
\end{align}
The fraction on the right-hand side depends only on $b$, $y$, and $\mu$, so we have a bound of the desired form.
\end{proof}

Given Theorems \ref{IRL:L1lower} and \ref{IRL:L1upper}, it is possible to specify a function space as the domain for the lower and upper incomplete Riemann--Liouville fractional integrals. We state the definitions formally as follows.

\begin{defn}
\label{IRL:lowerdefn}
Let $b>0$, $0<y<1$, and $\mu\in\mathbb{C}$ with $\mathrm{Re}(\mu)>0$. For any function $f:[0,b]\rightarrow\mathbb{C}$ which is $L^1$ on the subinterval $[0,yb]$, the $\mu$th \textbf{lower incomplete Riemann--Liouville fractional integral} of $f$ is defined by the equations
\begin{align*}
\prescript{RL}{0}I_x^{\mu}[f(x);y]&=\frac{1}{\Gamma(\mu)}\int_0^{yx}(x-t)^{\mu-1}f(t)\,\mathrm{d}t \\
&=\frac{x^{\mu}}{\Gamma(\mu)}\int_0^y(1-u)^{\mu-1}f(ux)\,\mathrm{d}u \\
&=\frac{x^{\mu}y}{\Gamma(\mu)}\int_0^1(1-wy)^{\mu-1}f(ywx)\,\mathrm{d}w,
\end{align*}
namely by precisely the existing equations \eqref{IRLdef:intlower1}--\eqref{IRLdef:intlower3}, with the sign of $\mu$ inverted so that we are considering the $\mu$th fractional integral instead of the $\mu$th fractional derivative.
\end{defn}

\begin{defn}
\label{IRL:upperdefn1}
Let $b>0$, $0<y<1$, and $\mu\in\mathbb{C}$ with $\mathrm{Re}(\mu)>1$. For any function $f\in L^1[0,b]$, the $\mu$th \textbf{upper incomplete Riemann--Liouville fractional integral} of $f$ is defined by the equations
\begin{align*}
\prescript{RL}{0}I_x^{\mu}\{f(x);y\}&=\frac{1}{\Gamma(\mu)}\int_{yx}^x(x-t)^{\mu-1}f(t)\,\mathrm{d}t \\
&=\frac{x^{\mu}}{\Gamma(\mu)}\int_y^1(1-u)^{\mu-1}f(ux)\,\mathrm{d}u \\
&=\frac{x^{\mu}y}{\Gamma(\mu)}\int_0^{1-y}v^{\mu-1}f((1-v)x)\,\mathrm{d}v,
\end{align*}
namely by precisely the existing equations \eqref{IRLdef:intlower1}--\eqref{IRLdef:intlower3}, with the sign of $\mu$ inverted so that we are considering the $\mu$th fractional integral instead of the $\mu$th fractional derivative.
\end{defn}

In order to define the upper incomplete RL fractional integral for $0<\mathrm{Re}(\mu)<1$, we need a different way of bounding the integral \eqref{IRLdef:intupper1}. This is provided by Theorem \ref{IRL:Linfty} below, after which we state another formal definition to accompany Definition \ref{IRL:upperdefn1}.

\begin{thm}
\label{IRL:Linfty}
If $b>0$ and $0<y<1$ and $\mu\in\mathbb{C}$ with $\mathrm{Re}(\mu)>0$, then the $\mu$th lower incomplete Riemann--Liouville fractional integral defines a bounded operator \[\prescript{RL}{0}D^{-\mu}[\cdot;y]:L^{\infty}[0,yb]\rightarrow L^{\infty}[0,b],\] and the $\mu$th upper incomplete Riemann--Liouville fractional integral defines a bounded operator \[\prescript{RL}{0}D^{-\mu}\{\cdot;y\}:L^{\infty}[0,b]\rightarrow L^{\infty}[0,b].\]
\end{thm}

\begin{proof}
Let $f$ be a function defined on $[0,b]$. We need to prove that the $L^{\infty}[0,b]$ norm of the function $\prescript{RL}{0}D_x^{-\mu}[f(x);y]$ is uniformly bounded in terms of the $L^{\infty}[0,yb]$ norm of $f$, and that the $L^{\infty}[0,b]$ norm of the function $\prescript{RL}{0}D_x^{-\mu}\{f(x);y\}$ is uniformly bounded in terms of the $L^{\infty}[0,b]$ norm of $f$.

\textbf{Case 1: lower incomplete.} We start from the definition \eqref{IRLdef:intlower1}. For any $x\in[0,b]$,
\begin{align*}
\Big|\prescript{RL}{0}D_x^{-\mu}[f(x);y]\Big|&\leq\frac{1}{|\Gamma(\mu)|}\int_0^{yx}|f(t)|(x-t)^{\mathrm{Re}(\mu)-1}\,\mathrm{d}t \\
&\leq\frac{1}{|\Gamma(\mu)|}\left(\esssup_{[0,yx]}|f|\right)\int_0^{yx}(x-t)^{\mathrm{Re}(\mu)-1}\,\mathrm{d}t \\
&=\frac{1}{\mathrm{Re}(\mu)|\Gamma(\mu)|}\esssup_{[0,yx]}|f|\left[(x-t)^{\mathrm{Re}(\mu)}\right]_{t=0}^{t=yx} \\
&=\frac{x^{\mathrm{Re}(\mu)}}{\mathrm{Re}(\mu)|\Gamma(\mu)|}\left[1-(1-y)^{\mathrm{Re}(\mu)}\right]\esssup_{[0,yx]}|f| \\
&\leq\frac{b^{\mathrm{Re}(\mu)}}{\mathrm{Re}(\mu)|\Gamma(\mu)|}\left[1-(1-y)^{\mathrm{Re}(\mu)}\right]\Big\|f\Big\|_{L^{\infty}[0,yb]}.
\end{align*}
Taking the supremum over all $x$, we deduce that
\begin{equation}
\label{IRL:Linftylower:eqn}
\Big\|\prescript{RL}{0}D^{-\mu}[f;y]\Big\|_{L^{\infty}[0,b]}\leq\frac{b^{\mathrm{Re}(\mu)}}{\mathrm{Re}(\mu)|\Gamma(\mu)|}\left[1-(1-y)^{\mathrm{Re}(\mu)}\right]\Big\|f\Big\|_{L^{\infty}[0,yb]}.
\end{equation}
The coefficient accompanying the norm on the right-hand side depends only on $b$, $y$, and $\mu$, so we have the desired result for lower incomplete RL integrals.

\textbf{Case 2: upper incomplete.} We start from the definition \eqref{IRLdef:intupper1}. For any $x\in[0,b]$,
\begin{align*}
\Big|\prescript{RL}{0}D_x^{-\mu}\{f(x);y\}\Big|&\leq\frac{1}{|\Gamma(\mu)|}\int_{yx}^x|f(t)|(x-t)^{\mathrm{Re}(\mu)-1}\,\mathrm{d}t \\
&\leq\frac{1}{|\Gamma(\mu)|}\left(\esssup_{[0,yx]}|f|\right)\int_{yx}^x(x-t)^{\mathrm{Re}(\mu)-1}\,\mathrm{d}t \\
&=\frac{1}{\mathrm{Re}(\mu)|\Gamma(\mu)|}\esssup_{[yx,x]}|f|\left[(x-t)^{\mathrm{Re}(\mu)}\right]_{t=yx}^{t=x} \\
&=\frac{x^{\mathrm{Re}(\mu)}}{\mathrm{Re}(\mu)|\Gamma(\mu)|}\left[(1-y)^{\mathrm{Re}(\mu)}-0\right]\esssup_{[yx,x]}|f| \\
&\leq\frac{b^{\mathrm{Re}(\mu)}(1-y)^{\mathrm{Re}(\mu)}}{\mathrm{Re}(\mu)|\Gamma(\mu)|}\Big\|f\Big\|_{L^{\infty}[0,b]}.
\end{align*}
(Note that here we have used the assumption that $\mathrm{Re}(\mu)>0$.) Taking the supremum over all $x$, we deduce that
\begin{equation}
\label{IRL:Linftyupper:eqn}
\Big\|\prescript{RL}{0}D^{-\mu}\{f;y\}\Big\|_{L^{\infty}[0,b]}\leq\frac{b^{\mathrm{Re}(\mu)}(1-y)^{\mathrm{Re}(\mu)}}{\mathrm{Re}(\mu)|\Gamma(\mu)|}\Big\|f\Big\|_{L^{\infty}[0,b]}.
\end{equation}
Again, the fraction on the right-hand side depends only on $b$, $y$, and $\mu$, so we have the desired result for upper incomplete RL integrals.
\end{proof}

Given the second part of Theorem \ref{IRL:Linfty}, it is possible to specify a function space as the domain for the upper incomplete Riemann--Liouville fractional integral even in the case $0<\mathrm{Re}(\mu)\leq1$. We state the definition formally as follows, to complement Definition \ref{IRL:upperdefn1}.

\begin{defn}
\label{IRL:upperdefn2}
Let $b>0$, $0<y<1$, and $\mu\in\mathbb{C}$ with $0<\mathrm{Re}(\mu)\leq1$. For any function $f\in L^{\infty}[0,b]$, the $\mu$th \textbf{upper incomplete Riemann--Liouville fractional integral} of $f$ is defined by the same equations as in Definition \ref{IRL:upperdefn1}, namely once again by \eqref{IRLdef:intupper1}--\eqref{IRLdef:intupper3} with the sign of $\mu$ inverted.
\end{defn}

Note that the restriction $\mathrm{Re}(\mu)\leq1$ is not required for Definition \ref{IRL:upperdefn2} to make sense. We include it only because the definition in the case $\mathrm{Re}(\mu)>1$ is already established, on a larger function space than $L^{\infty}[0,b]$, by the previous Definition \ref{IRL:upperdefn1}.

\begin{remark}
The nature of the domain of the lower incomplete RL fractional integral, as specified in the above theorems and definitions, is interesting because these operators allow us to extend the domain of good behaviour for $f$.

For example, if we start with a function $f:[0,b]\rightarrow\mathbb{C}$ which is $L^1$ only on the subinterval $[0,yb]$, then after applying the lower incomplete RL fractional integral, we obtain a new function which is $L^1$ on the whole of $[0,b]$. Similarly with $L^{\infty}$ or indeed, by H\"older's inequality, any other $L^p$ space.

Such extension of domains could be very important in the theory of partial differential equations, in which a well-behaved forcing function is used to prove regularity results for an unknown solution function \cite{fernandez,hormander}.
\end{remark}

In the case where $\mu$ is real, the inequalities bounding the operator norms for the incomplete RL integrals can be written in a more elegant form. We include this result as a corollary.

\begin{coroll}
Let $b>0$, $0<y<1$, $\mu\in\mathbb{R}^+$, and let $f$ be a function defined on $[0,b]$.
\begin{enumerate}
\item If $f\in L^1[0,yb]$ and $0<\mu\leq1$, then \[\Big\|\prescript{RL}{0}D^{-\mu}[f;y]\Big\|_{L^1[0,b]}\leq\frac{(1-y)^{\mu-1}b^{\mu}}{\Gamma(\mu+1)}\Big\|f(t)\Big\|_{L^1[0,yb]}.\]
\item If $f\in L^1[0,yb]$ and $\mu>1$, then \[\Big\|\prescript{RL}{0}D^{-\mu}[f;y]\Big\|_{L^1[0,b]}\leq\frac{b^{\mu}}{\Gamma(\mu+1)}\Big\|f(t)\Big\|_{L^1[0,yb]}.\]
\item If $f\in L^1[0,b]$ and $\mu>1$, then \[\Big\|\prescript{RL}{0}D^{-\mu}\{f;y\}\Big\|_{L^1[0,b]}\leq\frac{(1-y)^{\mu-1}b^{\mu}}{\Gamma(\mu+1)}\Big\|f(t)\Big\|_{L^1[0,yb]}.\]
\item If $f\in L^{\infty}[0,yb]$, then \[\Big\|\prescript{RL}{0}D^{-\mu}[f;y]\Big\|_{L^{\infty}[0,b]}\leq\frac{\left[1-(1-y)^{\mu}\right]b^{\mu}}{\Gamma(\mu+1)}\Big\|f\Big\|_{L^{\infty}[0,yb]}.\]
\item If $f\in L^{\infty}[0,b]$, then \[\Big\|\prescript{RL}{0}D^{-\mu}\{f;y\}\Big\|_{L^{\infty}[0,b]}\leq\frac{(1-y)^{\mu}b^{\mu}}{\Gamma(\mu+1)}\Big\|f\Big\|_{L^{\infty}[0,b]}.\]
\end{enumerate}
\end{coroll}

\begin{proof}
These results are just the inequalities \eqref{IRL:L1lower:eqn1}, \eqref{IRL:L1lower:eqn2}, \eqref{IRL:L1upper:eqn}, \eqref{IRL:Linftylower:eqn}, and \eqref{IRL:Linftyupper:eqn} in the case $\mu\in\mathbb{R}$.
\end{proof}

\begin{remark}
Letting $y\rightarrow0$ in the above inequalities for $L^1$ and $L^{\infty}$ norms of the lower incomplete RL integral yields some interesting results.

The inequality \eqref{IRL:L1lower:eqn1} is \[\Big\|\prescript{RL}{0}D^{-\mu}[f;y]\Big\|_{L^1[0,b]}\leq\frac{(1-y)^{\mathrm{Re}(\mu)-1}b^{\mathrm{Re}(\mu)}}{|\Gamma(\mu)|\mathrm{Re}(\mu)}\Big\|f(t)\Big\|_{L^1[0,yb]}.\] As $y\rightarrow0$, the right-hand side of this inequality tends to \[\frac{b^{\mathrm{Re}(\mu)}}{|\Gamma(\mu)|\mathrm{Re}(\mu)}\lim_{y\rightarrow0}\Big\|f(t)\Big\|_{L^1[0,yb]},\] which equals \[\frac{b^{\mathrm{Re}(\mu)}f(0)}{|\Gamma(\mu)|\mathrm{Re}(\mu)}\] if $0$ is a Lebesgue point of $f$.

The inequality \eqref{IRL:L1lower:eqn2} is \[\Big\|\prescript{RL}{0}D^{-\mu}[f;y]\Big\|_{L^1[0,b]}\leq\frac{b^{\mathrm{Re}(\mu)}}{|\Gamma(\mu)|\mathrm{Re}(\mu)}\Big\|f(t)\Big\|_{L^1[0,yb]}.\] As $y\rightarrow0$, the right-hand side of this inequality again tends to \[\frac{b^{\mathrm{Re}(\mu)}f(0)}{|\Gamma(\mu)|\mathrm{Re}(\mu)}\] if $0$ is a Lebesgue point of $f$.

The inequality \eqref{IRL:Linftylower:eqn} is \[\Big\|\prescript{RL}{0}D^{-\mu}[f;y]\Big\|_{L^{\infty}[0,b]}\leq\frac{\left[1-(1-y)^{\mathrm{Re}(\mu)}\right]b^{\mathrm{Re}(\mu)}}{\mathrm{Re}(\mu)|\Gamma(\mu)|}\Big\|f\Big\|_{L^{\infty}[0,yb]}.\] As $y\rightarrow0$, the right-hand side of this inequality tends asymptotically to \[\frac{\left[y\mathrm{Re}(\mu)\right]b^{\mathrm{Re}(\mu)}}{\mathrm{Re}(\mu)|\Gamma(\mu)|}\lim_{y\rightarrow0}\Big\|f\Big\|_{L^{\infty}[0,yb]},\] which yields the following leading-order linear term: \[\frac{yb^{\mathrm{Re}(\mu)}f(0)}{|\Gamma(\mu)|},\] if $0$ is a point of continuity of $f$.
\end{remark}

\subsection{Definitions for the fractional derivatives}

Fractional \textit{integrals} of incomplete Riemann--Liouville type were proposed in \cite{ozarslan-ustaoglu1} and their conditions carefully specified in the work above. What about fractional \textit{derivatives}? The Definitions \ref{IRL:lowerdefn}, \ref{IRL:upperdefn1}, and \ref{IRL:upperdefn2} are specified to define $\prescript{RL}{0}D_x^{\mu}[f(x);y]$ and $\prescript{RL}{0}D_x^{\mu}\{f(x);y\}$ only in the case $\mathrm{Re}(\mu)<0$, but for a fully developed model of fractional calculus it should also be possible to define these operators in the case $\mathrm{Re}(\mu)\geq0$.

In the classical Riemann--Liouville model, the fractional derivatives are defined by taking standard integer-order derivatives of appropriate fractional integrals. Thus, we might be tempted to do the same thing here, e.g. defining $\prescript{RL}{0}D_x^{1/2}[f(x);y]=\frac{\mathrm{d}}{\mathrm{d}x}\left(\prescript{RL}{0}D_x^{-1/2}[f(x);y]\right)$ and $\prescript{RL}{0}D_x^{1/2}\{f(x);y\}=\frac{\mathrm{d}}{\mathrm{d}x}\left(\prescript{RL}{0}D_x^{-1/2}\{f(x);y\}\right)$. This also seems like a natural complement to the existing definition for incomplete Caputo fractional derivatives \cite{ozarslan-ustaoglu2}. However, it is not clear whether or not this would be a `natural' extension of the Definitions \ref{IRL:lowerdefn}, \ref{IRL:upperdefn1}, and \ref{IRL:upperdefn2}.

The obvious question to ask, then, is: what makes the Riemann--Liouville derivatives a `natural' extension of the definition of Riemann--Liouville integrals? What is the justification for this definition over, say, that of Caputo derivatives?

One answer to this question is that the Riemann--Liouville fractional derivative $\prescript{RL}{c}D^{\mu}_xf(x),\mathrm{Re}(\mu)\geq0,$ forms the \textbf{analytic continuation in $\boldsymbol{\mu}$} of the Riemann--Liouville fractional integral $\prescript{RL}{c}D^{\mu}_xf(x),\mathrm{Re}(\mu)<0$. This way of thinking is unique to fractional calculus: with $\mu$ as a continuous variable, it is possible to perform calculus with respect to $\mu$ as well as with respect to $x$.

More specifically, if we define a function $F_x$ by \[F_x(\mu)=\prescript{RL}{c}D^{\mu}_xf(x),\quad\quad\mathrm{Re}(\mu)<0,\] then this function is analytic and satisfies the following functional equation:
\begin{equation}
\label{RL:fnleqn}
\frac{\mathrm{d}}{\mathrm{d}x}F_x(\mu)=F_x(\mu+1),\quad\quad\mathrm{Re}(\mu)<-1.
\end{equation}
This can then be used to extend $F_x$ to a meromorphic function on the entire complex plane. The functional equation \eqref{RL:fnleqn} gives us a way of defining $F_x(\mu)$ for $0\leq\mathrm{Re}(\mu)<1$, then for $1\leq\mathrm{Re}(\mu)<2$, then for $2\leq\mathrm{Re}(\mu)<3$, etc., in such a way that it is analytic on each of these regions. This analytic continuation is precisely the Riemann--Liouville fractional derivative.

Can we similarly use analytic continuation to define upper and lower incomplete Riemann--Liouville fractional derivatives? In order to find an analogue of the functional equation \eqref{RL:fnleqn}, we must consider the effect of the differentiation operator on the upper and lower incomplete Riemann--Liouville fractional integrals . To this end, the following two theorems are established.

\begin{thm}
\label{composn:lower}
The composition of the lower incomplete Riemann--Liouville fractional integral with the standard operation of differentiation is given by the following identities:
\begin{align}
\label{composn:d:lower} \frac{\mathrm{d}}{\mathrm{d}x}\Big(\prescript{RL}{0}D_x^{-\mu}[f(x);y]\Big)&=\frac{y(1-y)^{\mu-1}}{\Gamma(\mu)}x^{\mu-1}f(xy)+\prescript{RL}{0}D^{1-\mu}_x[f(x);y], \\
\label{composn:lower:d} \prescript{RL}{0}D_x^{-\mu}[f'(x);y]&=\frac{x^{\mu-1}}{\Gamma(\mu)}\left((1-y)^{\mu-1}f(xy)-f(0)\right)+\prescript{RL}{0}D^{1-\mu}_x[f(x);y],
\end{align}
valid for $\mathrm{Re}(\mu)>1$ and for $f,x,y$ satisfying the appropriate criteria from Definition \ref{IRL:lowerdefn}.
\end{thm}

\begin{proof}
To prove \eqref{composn:d:lower}, we start from the definition \eqref{IRLdef:intlower1} and use the standard method for differentiating with respect to $x$ an integral expression whose $x$-dependence is both in the integrand and in the upper bound of integration:
\begin{align*}
\frac{\mathrm{d}}{\mathrm{d}x}\Big(\prescript{RL}{0}D_x^{-\mu}[f(x);y]\Big)&=\frac{\mathrm{d}}{\mathrm{d}x}\left(\frac{1}{\Gamma(\mu)}\int_0^{yx}(x-t)^{\mu-1}f(t)\,\mathrm{d}t\right) \\
&=\frac{1}{\Gamma(\mu)}\left(y(x-yx)^{\mu-1}f(yx)+\int_0^{yx}(\mu-1)(x-t)^{\mu-2}f(t)\,\mathrm{d}t\right) \\
&=\frac{y(1-y)^{\mu-1}x^{\mu-1}f(yx)}{\Gamma(\mu)}+\frac{\mu-1}{\Gamma(\mu)}\int_0^{yx}(x-t)^{\mu-2}f(t)\,\mathrm{d}t \\
&=\frac{y(1-y)^{\mu-1}}{\Gamma(\mu)}x^{\mu-1}f(yx)+\prescript{RL}{0}D^{1-\mu}_x[f(x);y],
\end{align*}
as required, where for the final step we used the fact that $\Gamma(\mu)=(\mu-1)\Gamma(\mu-1)$.

To prove \eqref{composn:lower:d}, we again start from the definition \eqref{IRLdef:intlower1} and use integration by parts:
\begin{align*}
\prescript{RL}{0}D_x^{-\mu}[f'(x);y]&=\frac{1}{\Gamma(\mu)}\int_0^{yx}(x-t)^{\mu-1}f'(t)\,\mathrm{d}t \\
&=\frac{1}{\Gamma(\mu)}\left(\Big[(x-t)^{\mu-1}f(t)\Big]_{t=0}^{t=yx}+\int_0^{yx}(\mu-1)(x-t)^{\mu-2}f(t)\,\mathrm{d}t\right) \\
&=\frac{1}{\Gamma(\mu)}\Big((x-yx)^{\mu-1}f(xy)-x^{\mu-1}f(0)\Big)+\frac{\mu-1}{\Gamma(\mu)}\int_0^{yx}(x-t)^{\mu-2}f(t)\,\mathrm{d}t \\
&=\frac{x^{\mu-1}}{\Gamma(\mu)}\left((1-y)^{\mu-1}f(xy)-f(0)\right)+\prescript{RL}{0}D^{1-\mu}_x[f(x);y],
\end{align*}
as required, where again we used $\Gamma(\mu)=(\mu-1)\Gamma(\mu-1)$ in the final step.
\end{proof}

\begin{thm}
\label{composn:upper}
The composition of the upper incomplete Riemann--Liouville fractional integral with the standard operation of differentiation is given by the following identities:
\begin{align}
\label{composn:d:upper} \frac{\mathrm{d}}{\mathrm{d}x}\Big(\prescript{RL}{0}D_x^{-\mu}\{f(x);y\}\Big)&=-\frac{y(1-y)^{\mu-1}}{\Gamma(\mu)}x^{\mu-1}f(xy)+\prescript{RL}{0}D^{1-\mu}_x\{f(x);y\}, \\
\label{composn:upper:d} \prescript{RL}{0}D_x^{-\mu}\{f'(x);y\}&=-\frac{x^{\mu-1}}{\Gamma(\mu)}(1-y)^{\mu-1}f(xy)+\prescript{RL}{0}D^{1-\mu}_x\{f(x);y\},
\end{align}
valid for $\mathrm{Re}(\mu)>1$ and for $f,x,y$ satisfying the appropriate criteria from Definitions \ref{IRL:upperdefn1} and \ref{IRL:upperdefn2}.
\end{thm}

\begin{proof}
To prove \eqref{composn:d:upper}, we start from the definition \eqref{IRLdef:intupper1} and use the standard method for differentiating with respect to $x$ an integral expression whose $x$-dependence is in the integrand and in both bounds of integration:
\begin{align*}
\frac{\mathrm{d}}{\mathrm{d}x}\Big(\prescript{RL}{0}D_x^{-\mu}\{f(x);y\}\Big)&=\frac{\mathrm{d}}{\mathrm{d}x}\left(\frac{1}{\Gamma(\mu)}\int_{yx}^x(x-t)^{\mu-1}f(t)\,\mathrm{d}t\right) \\
&=\frac{1}{\Gamma(\mu)}\left((x-x)^{\mu-1}f(x)-y(x-yx)^{\mu-1}f(yx)+\int_{yx}^x(\mu-1)(x-t)^{\mu-2}f(t)\,\mathrm{d}t\right) \\
&=\frac{-y(1-y)^{\mu-1}x^{\mu-1}f(yx)}{\Gamma(\mu)}+\frac{\mu-1}{\Gamma(\mu)}\int_{yx}^x(x-t)^{\mu-2}f(t)\,\mathrm{d}t \\
&=\frac{-y(1-y)^{\mu-1}}{\Gamma(\mu)}x^{\mu-1}f(yx)+\prescript{RL}{0}D^{1-\mu}_x\{f(x);y\},
\end{align*}
as required, where in the third line we used the assumption that $\mathrm{Re}(\mu)>1$.

To prove \eqref{composn:upper:d}, we again start from the definition \eqref{IRLdef:intupper1} and use integration by parts:
\begin{align*}
\prescript{RL}{0}D_x^{-\mu}\{f'(x);y\}&=\frac{1}{\Gamma(\mu)}\int_{yx}^x(x-t)^{\mu-1}f'(t)\,\mathrm{d}t \\
&=\frac{1}{\Gamma(\mu)}\left(\Big[(x-t)^{\mu-1}f(t)\Big]_{t=yx}^{t=x}+\int_{yx}^x(\mu-1)(x-t)^{\mu-2}f(t)\,\mathrm{d}t\right) \\
&=\frac{1}{\Gamma(\mu)}\Big((x-x)^{\mu-1}f(x)-(x-yx)^{\mu-1}f(xy)\Big)+\frac{\mu-1}{\Gamma(\mu)}\int_{yx}^x(x-t)^{\mu-2}f(t)\,\mathrm{d}t \\
&=\frac{x^{\mu-1}}{\Gamma(\mu)}\left(-(1-y)^{\mu-1}f(xy)\right)+\prescript{RL}{0}D^{1-\mu}_x\{f(x);y\},
\end{align*}
as required, where again we used $\mathrm{Re}(\mu)>1$ in the third line.
\end{proof}

The above Theorems \ref{composn:lower} and \ref{composn:upper} can be used, in the same way as discussed at the start of this section, to construct analytic continuations of $\prescript{RL}{0}D^{\mu}_x[f(x);y]$ and $\prescript{RL}{0}D^{\mu}_x\{f(x);y\}$ which are valid for $\mathrm{Re}(\mu)\geq0$ (fractional derivatives) as well as for $\mathrm{Re}(\mu)<0$ (fractional integrals). The definitions are stated formally in Definitions \ref{IRLdef:derlower} and \ref{IRLdef:derupper}.

It is important to note that the existing formulae for $\prescript{RL}{0}D^{\mu}_x[f(x);y]$ and $\prescript{RL}{0}D^{\mu}_x\{f(x);y\}$ given by Definitions \ref{IRL:lowerdefn}, \ref{IRL:upperdefn1}, and \ref{IRL:upperdefn2} are analytic on the open left half-plane as functions of the complex variable $\mu$. Thus, the concept of ``analytic continuation'' outside of this domain makes sense.

\begin{defn}
\label{IRLdef:derlower}
The $\mu$th \textbf{lower incomplete Riemann--Liouville fractional derivative} of a function $f$ is defined by using the equation \eqref{composn:d:lower} for each successive region
\begin{equation}
\label{regions}
0\leq\mathrm{Re}(\mu)<1\quad,\quad1\leq\mathrm{Re}(\mu)<2\quad,\quad2\leq\mathrm{Re}(\mu)<3\quad,\quad\dots
\end{equation}
In other words, we define
\begin{equation}
\label{IRLdef:derlower1}
\prescript{RL}{0}D^{\mu}_x[f(x);y]=\frac{\mathrm{d}}{\mathrm{d}x}\Big(\prescript{RL}{0}D_x^{\mu-1}[f(x);y]\Big)-\frac{y(1-y)^{-\mu}}{\Gamma(1-\mu)}x^{-\mu}f(xy),
\end{equation}
for $\mu$ in each of the regions \eqref{regions} successively, and thence on the entire half-plane $\mathrm{Re}(\mu)\geq0$.
\end{defn}

\begin{defn}
\label{IRLdef:derupper}
The $\mu$th \textbf{upper incomplete Riemann--Liouville fractional derivative} of a function $f$ is defined by using the equation \eqref{composn:d:upper} for $\mu$ in each of the regions \eqref{regions} successively. In other words, we define
\begin{equation}
\label{IRLdef:derupper1}
\prescript{RL}{0}D^{\mu}_x\{f(x);y\}=\frac{\mathrm{d}}{\mathrm{d}x}\Big(\prescript{RL}{0}D_x^{\mu-1}\{f(x);y\}\Big)+\frac{y(1-y)^{-\mu}}{\Gamma(1-\mu)}x^{-\mu}f(xy),
\end{equation}
to get an analytic continuation to the entire half-plane $\mathrm{Re}(\mu)\geq0$.
\end{defn}

The above work has established that it is possible to define fractional derivatives as well as fractional integrals in the incomplete Riemann--Liouville context. However, they would still be difficult to compute when $\mathrm{Re}(\mu)$ is large, requiring many iterations of the equations \eqref{IRLdef:derlower1} and \eqref{IRLdef:derupper1}. It is much easier to use the direct formulae given by the following theorems.

\begin{thm}
\label{IRLupper:RL}
We have the following exact equivalence, valid for all $\mu\in\mathbb{C}$ and all functions $f$ such that the operators are defined:
\begin{equation}
\label{IRLupper:RL:eqn}
\prescript{RL}{0}D_x^{\mu}\{f(x);y\}=\prescript{RL}{xy}D_x^{\mu}f(x).
\end{equation}
\end{thm}

\begin{proof}
For $\mathrm{Re}(\mu)<0$, this follows immediately from the integral definitions of the operators. Starting from the formula \eqref{IRLdef:intupper1}, we have:
\begin{equation*}
\prescript{RL}{0}D_x^{\mu}\{f(x);y\}=\frac{1}{\Gamma(-\mu)}\int_{yx}^x(x-t)^{-\mu-1}f(t)\,\mathrm{d}t=\prescript{RL}{xy}D_x^{\mu}f(x).
\end{equation*}
Having proved the result for $\mathrm{Re}(\mu)<0$, we can now extend it to all $\mu\in\mathbb{C}$ by analytic continuation, since both sides of \eqref{IRLupper:RL:eqn} are analytic as functions of $\mu$.
\end{proof}

\begin{remark}
Note that the result of Theorem \ref{IRLupper:RL} does not mean the upper incomplete RL operator is just a special case of the usual RL operator. The theory is different in the incomplete case, due to the $x$-dependence appearing in a new place in the expression. The result is important, but it does not reduce incomplete RL fractional calculus to merely a subset of RL fractional calculus.

For example, although it is true that \[\frac{\mathrm{d}}{\mathrm{d}x}\left(\prescript{RL}{c}D^{\mu}f(x)\right)=\prescript{RL}{c}D^{\mu+1}f(x)\] for any constant $c$, this result is \textit{not} true when $c$ is replaced by $xy$ as in \eqref{IRLupper:RL:eqn}. Instead, we have the differentiation relation \eqref{composn:d:upper} which was already proved in Theorem \ref{composn:upper}. Or again, although the operator $\prescript{RL}{c}D^{\mu}$ has a semigroup property in $\mu$ for any constant $c$, the operator $\prescript{RL}{xy}D^{\mu}$ does not. (We explore the semigroup property for our operators more thoroughly in Section \ref{sec:further} below.)
\end{remark}

\begin{thm}
\label{IRLlower:extension}
The formulae \eqref{IRLdef:intlower1}--\eqref{IRLdef:intlower3} are valid 
expressions for $\prescript{RL}{0}D_x^{-\mu}[f(x);y]$ for all $\mu\in\mathbb{C}$, not only for $\mathrm{Re}(\mu)>0$.
\end{thm}

\begin{proof}
The restriction $\mathrm{Re}(\mu)>0$ was never actually required for these formulae. It is required for the definition of the usual Riemann--Liouville integral, because the integrand of $\int_0^{x}(x-t)^{-\mu-1}f(t)\,\mathrm{d}t$ has a singularity at $t=x$. But when the integral is restricted to $[0,yx]$ instead of $[0,x]$, this singularity is no longer part of the domain. The same argument holds for each of the integrals in \eqref{IRLdef:intlower1}--\eqref{IRLdef:intlower3}: respectively, the points $t=x$ in \eqref{IRLdef:intlower1}, $u=1$ in \eqref{IRLdef:intlower2}, and $w=\frac{1}{y}$ in \eqref{IRLdef:intlower3} are excluded from the domain of integration.
\end{proof}

The importance of Theorems \ref{IRLupper:RL} and \ref{IRLlower:extension} is that they are easier to use and apply than Definitions \ref{IRLdef:derlower} and \ref{IRLdef:derupper} as expressions for the upper and lower incomplete RL derivatives. For the incomplete RL integrals, we already have the original formulae \eqref{IRLdef:intlower1}--\eqref{IRLdef:intlower3} and \eqref{IRLdef:intupper1}--\eqref{IRLdef:intupper3} which can be applied as in the original RL model; but for the incomplete RL derivatives, it is much easier to use the formulae \eqref{IRLdef:intlower1}--\eqref{IRLdef:intlower3} and \eqref{IRLupper:RL:eqn} than iterations of the formulae \eqref{IRLdef:derlower1} and \eqref{IRLdef:derupper1}.

As examples to illustrate the above theorems, we compute the incomplete fractional derivatives of some simple functions, and verify that all the formulae considered above are consistent.

\begin{example}
\label{power1}
We consider the function $f(x)=x^{\lambda}$. It is known \cite[Theorems 19--20]{ozarslan-ustaoglu1} that the incomplete fractional integrals of this function are given by
\begin{align}
\label{power1:intlower} \prescript{RL}{0}D^{\mu}_x[x^{\lambda};y]&=\frac{B_y(\lambda+1,-\mu)}{\Gamma(-\mu)}x^{\lambda-\mu},\quad\quad\mathrm{Re}(\lambda)>-1,\mathrm{Re}(\mu)<0;\\
\label{power1:intupper} \prescript{RL}{0}D^{\mu}_x\{x^{\lambda};y\}&=\frac{B_{1-y}(-\mu,\lambda+1)}{\Gamma(-\mu)}x^{\lambda-\mu},\quad\quad\mathrm{Re}(\lambda)>-1,\mathrm{Re}(\mu)<0.
\end{align}
By analytic continuation, we expect that the same expressions \eqref{power1:intlower} and \eqref{power1:intupper} will be valid for all $\mu\in\mathbb{C}$, i.e. for fractional derivatives as well as fractional integrals. This can be verified using Definitions \ref{IRLdef:derlower} and \ref{IRLdef:derupper}, as follows. %if we can prove the following identity for $\mathrm{Re}(\mu)>1$ (it will then be valid for all $\mu$ by analytic continuation):

Firstly, lower incomplete. For $0\leq\mathrm{Re}(\mu)<1$, we substitute the known expression \eqref{power1:intlower} for $\prescript{RL}{0}D^{\mu-1}_x[x^{\lambda};y]$ into the identity \eqref{IRLdef:derlower1} to get:
\begin{align*}
\prescript{RL}{0}D^{\mu}_x[f(x);y]&=\frac{\mathrm{d}}{\mathrm{d}x}\Big(\prescript{RL}{0}D_x^{\mu-1}[f(x);y]\Big)-\frac{y(1-y)^{-\mu}}{\Gamma(1-\mu)}x^{-\mu}f(xy) \\
&=\frac{\mathrm{d}}{\mathrm{d}x}\left(\frac{B_y(\lambda+1,1-\mu)}{\Gamma(1-\mu)}x^{\lambda-\mu+1}\right)-\frac{y(1-y)^{-\mu}}{\Gamma(1-\mu)}x^{-\mu}(xy)^{\lambda} \\
&=(\lambda-\mu+1)\frac{B_y(\lambda+1,1-\mu)}{\Gamma(1-\mu)}x^{\lambda-\mu}-\frac{y^{\lambda+1}(1-y)^{-\mu}}{\Gamma(1-\mu)}x^{\lambda-\mu} \\
&=\frac{(\lambda-\mu+1)B_y(\lambda+1,1-\mu)-y^{\lambda+1}(1-y)^{-\mu}}{\Gamma(1-\mu)}x^{\lambda-\mu}.
\end{align*}
The following is a natural property of the incomplete beta function, following from integration by parts applied to the defining integrals: \[(\lambda-\mu+1)B_y(\lambda+1,1-\mu)-y^{\lambda+1}(1-y)^{-\mu}=-\mu B_y(\lambda+1,-\mu).\] This confirms the expression \eqref{power1:intlower} for the lower incomplete derivative when $0\leq\mathrm{Re}(\mu)<1$.

The same argument works to confirm it for $1\leq\mathrm{Re}(\mu)<2$, $2\leq\mathrm{Re}(\mu)<3$, etc., since there was no assumption on the value of $\mu$ in the above manipulations of incomplete beta functions. Thus, as expected, \eqref{power1:intlower} is valid for all $\mu\in\mathbb{C}$.

Secondly, upper incomplete. For $0\leq\mathrm{Re}(\mu)<1$, we substitute the known expression \eqref{power1:intupper} for $\prescript{RL}{0}D^{\mu-1}_x\{x^{\lambda};y\}$ into the identity \eqref{IRLdef:derupper1} to get:
\begin{align*}
\prescript{RL}{0}D^{\mu}_x\{f(x);y\}&=\frac{\mathrm{d}}{\mathrm{d}x}\Big(\prescript{RL}{0}D_x^{\mu-1}\{f(x);y\}\Big)+\frac{y(1-y)^{-\mu}}{\Gamma(1-\mu)}x^{-\mu}f(xy) \\
&=\frac{\mathrm{d}}{\mathrm{d}x}\left(\frac{B_{1-y}(1-\mu,\lambda+1)}{\Gamma(1-\mu)}x^{\lambda-\mu+1}\right)+\frac{y(1-y)^{-\mu}}{\Gamma(1-\mu)}x^{-\mu}(xy)^{\lambda} \\
&=(\lambda-\mu+1)\frac{B_{1-y}(1-\mu,\lambda+1)}{\Gamma(1-\mu)}x^{\lambda-\mu}+\frac{y^{\lambda+1}(1-y)^{-\mu}}{\Gamma(1-\mu)}x^{\lambda-\mu} \\
&=\frac{(\lambda-\mu+1)B_{1-y}(1-\mu,\lambda+1)+y^{\lambda+1}(1-y)^{-\mu}}{\Gamma(1-\mu)}x^{\lambda-\mu}.
\end{align*}
As before, it is a natural property of the incomplete beta function that \[(\lambda-\mu+1)B_{1-y}(1-\mu,\lambda+1)+y^{\lambda+1}(1-y)^{-\mu}=-\mu B_{1-y}(-\mu,\lambda+1).\] This confirms the expression \eqref{power1:intupper} for the upper incomplete derivative when $0\leq\mathrm{Re}(\mu)<1$.

The same argument works to confirm it for $1\leq\mathrm{Re}(\mu)<2$, $2\leq\mathrm{Re}(\mu)<3$, etc., since there was no assumption on the value of $\mu$ in the above manipulations of incomplete beta functions. Thus, as expected, \eqref{power1:intupper} is valid for all $\mu\in\mathbb{C}$. \qed
\end{example}

\begin{example}
\label{power2}

We consider the function $f(x)=x^{\lambda-1}(1-x)^{-\alpha}$. The incomplete fractional integrals of this function can be computed using the definitions \eqref{IRLdef:intlower2} and \eqref{IRLdef:intupper2}:
\begin{align*}
\prescript{RL}{0}D^{\mu}_x[x^{\lambda-1}(1-x)^{-\alpha};y]&=\frac{x^{-\mu}}{\Gamma(-\mu)}\int_0^y(1-u)^{-\mu-1}(ux)^{\lambda-1}(1-ux)^{-\alpha}\,\mathrm{d}u \\
&=\frac{x^{\lambda-\mu-1}}{\Gamma(-\mu)}\int_0^y(1-u)^{-\mu-1}(u)^{\lambda-1}(1-ux)^{-\alpha}\,\mathrm{d}u; \\
\prescript{RL}{0}D^{\mu}_x\{x^{\lambda-1}(1-x)^{-\alpha};y\}&=\frac{x^{-\mu}}{\Gamma(-\mu)}\int_y^1(1-u)^{-\mu-1}(ux)^{\lambda-1}(1-ux)^{-\alpha}\,\mathrm{d}u \\
&=\frac{x^{\lambda-\mu-1}}{\Gamma(-\mu)}\int_y^1(1-u)^{-\mu-1}(u)^{\lambda-1}(1-ux)^{-\alpha}\,\mathrm{d}u.
\end{align*}
Using the integral expressions for the incomplete hypergeometric functions, namely \cite[Eq. (27)]{ozarslan-ustaoglu1} for lower incomplete and its analogue for upper incomplete, we can rewrite these as follows:
\begin{align}
\nonumber \prescript{RL}{0}D^{\mu}_x[x^{\lambda-1}(1-x)^{-\alpha};y]&=\frac{x^{\lambda-\mu-1}}{\Gamma(-\mu)}B(\lambda,-\mu)\prescript{}{2}F_1(\alpha,[\lambda,\lambda-\mu;y];x) \\
\label{power2:intlower} &=\frac{\Gamma(\lambda)}{\Gamma(\lambda-\mu)}x^{\lambda-\mu-1}\prescript{}{2}F_1(\alpha,[\lambda,\lambda-\mu;y];x); \\
\nonumber \prescript{RL}{0}D^{\mu}_x\{x^{\lambda-1}(1-x)^{-\alpha};y\}&=\frac{x^{\lambda-\mu-1}}{\Gamma(-\mu)}B(\lambda,-\mu)\prescript{}{2}F_1(\alpha,\{\lambda,\lambda-\mu;y\};x) \\
\label{power2:intupper} &=\frac{\Gamma(\lambda)}{\Gamma(\lambda-\mu)}x^{\lambda-\mu-1}\prescript{}{2}F_1(\alpha,\{\lambda,\lambda-\mu;y\};x).
\end{align}
These identities are valid for $\mathrm{Re}(\mu)<0$, $\mathrm{Re}(\lambda)>0$, $\mathrm{Re}(\alpha)>0$, and $|x|<1$. By analytic continuation, we expect that the same expressions \eqref{power2:intlower} and \eqref{power2:intupper} should be valid for all $\mu\in\mathbb{C}$, i.e. for fractional derivatives as well as fractional integrals. Using Definitions \ref{IRLdef:derlower} and \ref{IRLdef:derupper}, we can argue as follows.

Firstly, lower incomplete. For $0\leq\mathrm{Re}(\mu)<1$, we substitute the known expression \eqref{power2:intlower} for $\prescript{RL}{0}D^{\mu-1}_x[x^{\lambda};y]$ into the identity \eqref{IRLdef:derlower1} to get:
\begin{align*}
\prescript{RL}{0}D^{\mu}_x[f(x);y]&=\frac{\mathrm{d}}{\mathrm{d}x}\Big(\prescript{RL}{0}D_x^{\mu-1}[f(x);y]\Big)-\frac{y(1-y)^{-\mu}}{\Gamma(1-\mu)}x^{-\mu}f(xy) \\
&=\frac{\mathrm{d}}{\mathrm{d}x}\left(\frac{\Gamma(\lambda)}{\Gamma(\lambda-\mu+1)}x^{\lambda-\mu}\prescript{}{2}F_1(\alpha,[\lambda,\lambda-\mu+1;y];x)\right)-\frac{y(1-y)^{-\mu}}{\Gamma(1-\mu)}x^{-\mu}(xy)^{\lambda-1}(1-xy)^{-\alpha} \\
&=\frac{\Gamma(\lambda)}{\Gamma(\lambda-\mu+1)}(\lambda-\mu)x^{\lambda-\mu-1}\prescript{}{2}F_1(\alpha,[\lambda,\lambda-\mu+1;y];x) \\
&\hspace{2cm}+\frac{\Gamma(\lambda)}{\Gamma(\lambda-\mu+1)}x^{\lambda-\mu}\left(\frac{\alpha\lambda}{\lambda-\mu+1}\right)\prescript{}{2}F_1(\alpha+1,[\lambda+1,\lambda-\mu+2;y];x) \\
&\hspace{2cm}-\frac{y^{\lambda}(1-y)^{-\mu}}{\Gamma(1-\mu)}x^{\lambda-\mu-1}(1-xy)^{-\alpha} \\
&=\frac{\Gamma(\lambda)}{\Gamma(\lambda-\mu)}x^{\lambda-\mu-1}\bigg[\prescript{}{2}F_1(\alpha,[\lambda,\lambda-\mu+1;y];x) \\
&\hspace{1cm}+\frac{\alpha\lambda}{(\lambda-\mu)(\lambda-\mu+1)}x\prescript{}{2}F_1(\alpha+1,[\lambda+1,\lambda-\mu+2;y];x)+\frac{y^{\lambda}(1-y)^{-\mu}}{\mu B(\lambda,-\mu)}(1-xy)^{-\alpha}\bigg] \\
&=\frac{\Gamma(\lambda)}{\Gamma(\lambda-\mu)}x^{\lambda-\mu-1}\prescript{}{2}F_1(\alpha,[\lambda,\lambda-\mu;y];x),
\end{align*}
where we have used identities from \cite[Theorems 12--13]{ozarslan-ustaoglu1} to simplify the expressions involving the incomplete hypergeometric function. This confirms the expression \eqref{power2:intlower} for the lower incomplete derivative when $0\leq\mathrm{Re}(\mu)<1$.

The same argument works to confirm it for $1\leq\mathrm{Re}(\mu)<2$, $2\leq\mathrm{Re}(\mu)<3$, etc., since there was no assumption on the value of $\mu$ in the above manipulations of incomplete hypergeometric functions. Thus, as expected, \eqref{power2:intlower} is valid for all $\mu\in\mathbb{C}$.

For the upper incomplete case, we can deduce \eqref{power2:intupper} from \eqref{power2:intlower} using the fact that their sum is the usual Riemnn--Liouville fractional differintegral which is well known \cite{miller-ross}. \qed
\end{example}

\begin{remark}
\label{inversion}
Given Examples \ref{power1} and \ref{power2}, we can immediately verify that the derivative and integral operators we have defined do not have inverse properties. For instance, applying an incomplete fractional integral and then an incomplete fractional derivative to a simple power function yields the following:
\begin{align*}
\prescript{RL}{0}D^{\mu}_x\Big[\prescript{RL}{0}I^{\mu}_x[x^{\lambda};y];y\Big]&=\prescript{RL}{0}D^{\mu}_x\left[\frac{B_y(\lambda+1,\mu)}{\Gamma(\mu)}x^{\lambda+\mu};y\right] \\
&=\frac{B_y(\lambda+1,\mu)}{\Gamma(\mu)}\prescript{RL}{0}D^{\mu}_x[x^{\lambda+\mu};y] \\
&=\frac{B_y(\lambda+1,\mu)}{\Gamma(\mu)}\left(\frac{B_y(\lambda+\mu+1,-\mu)}{\Gamma(-\mu)}x^{\lambda+\mu-\mu}\right) \\
&=\frac{B_y(\lambda+1,\mu)B_y(\lambda+\mu+1,-\mu)}{\Gamma(\mu)\Gamma(-\mu)}x^{\lambda}; \\
\prescript{RL}{0}D^{\mu}_x\Big\{\prescript{RL}{0}I^{\mu}_x\{x^{\lambda};y\};y\Big\}&=\prescript{RL}{0}D^{\mu}_x\left\{\frac{B_{1-y}(\mu,\lambda+1)}{\Gamma(\mu)}x^{\lambda+\mu};y\right\} \\
&=\frac{B_{1-y}(\mu,\lambda+1)}{\Gamma(\mu)}\prescript{RL}{0}D^{\mu}_x\{x^{\lambda+\mu};y\} \\
&=\frac{B_{1-y}(\mu,\lambda+1)}{\Gamma(\mu)}\left(\frac{B_{1-y}(-\mu,\lambda+\mu+1)}{\Gamma(-\mu)}x^{\lambda+\mu-\mu}\right) \\
&=\frac{B_{1-y}(\mu,\lambda+1)B_{1-y}(-\mu,\lambda+\mu+1)}{\Gamma(\mu)\Gamma(-\mu)}x^{\lambda}.
\end{align*}
Since neither $B_y(\lambda+1,\mu)B_y(\lambda+\mu+1,-\mu)$ nor $B_{1-y}(\mu,\lambda+1)B_{1-y}(-\mu,\lambda+\mu+1)$ are identically equal to $\Gamma(\mu)\Gamma(-\mu)$, we surmise that the incomplete fractional derivatives are not left inverses to the incomplete fractional integrals. This is one disadvantage of Definitions \ref{IRLdef:derlower} and \ref{IRLdef:derupper}, but it is counterbalanced by the advantages of a unified differintegral formula given by the analytic continuation method.
\end{remark}

\section{Further properties of incomplete Riemann--Liouville fractional calculus} \label{sec:further}

The previous section established rigorous definitions for incomplete Riemann--Liouville fractional calculus, by specifying function spaces on which the operators act, and defining fractional derivatives as well as fractional integrals in this model.

In the current section, we shall investigate further properties and results concerning these operators. Since the theory of incomplete Riemann--Liouville fractional calculus is still very new, there are many important properties which have yet to be examined, and useful theorems which have yet to be proved.

One fundamental question in any model of fractional calculus is whether the operators satisfy a \textbf{semigroup property}. In the standard Riemann--Liouville model, for example, the fractional integrals have a semigroup property while the fractional derivatives do not \cite{miller-ross,samko-kilbas-marichev}. What happens in the incomplete Riemann--Liouville model?

We have already seen in Remark \ref{inversion} that the incomplete fractional derivatives and integrals lack an inversion property, which would be a special case of the semigroup property for composition of fractional differintegral operators. A simple example is enough to verify that the semigroup property is not valid either for combinations of fractional integrals or for combinations of fractional derivatives:
\begin{align*}
\prescript{RL}{0}I^{\mu}_x\Big[\prescript{RL}{0}I^{\nu}_x[x^{\lambda};y];y\Big]&=\prescript{RL}{0}I^{\mu}_x\left[\frac{B_y(\lambda+1,\nu)}{\Gamma(\nu)}x^{\lambda+\nu};y\right] \\
&=\frac{B_y(\lambda+1,\nu)}{\Gamma(\nu)}\prescript{RL}{0}I^{\mu}_x[x^{\lambda+\nu};y] \\
&=\frac{B_y(\lambda+1,\nu)}{\Gamma(\nu)}\left(\frac{B_y(\lambda+\nu+1,\mu)}{\Gamma(\mu)}x^{\lambda+\mu+\nu}\right) \\
&=\frac{B_y(\lambda+1,\nu)B_y(\lambda+\nu+1,\mu)}{\Gamma(\mu)\Gamma(\nu)}x^{\lambda+\mu+\nu}; \\
\prescript{RL}{0}D^{\mu}_x\Big[\prescript{RL}{0}D^{\nu}_x[x^{\lambda};y];y\Big]&=\prescript{RL}{0}D^{\mu}_x\left[\frac{B_y(\lambda+1,-\nu)}{\Gamma(-\nu)}x^{\lambda-\nu};y\right] \\
&=\frac{B_y(\lambda+1,-\nu)}{\Gamma(-\nu)}\prescript{RL}{0}D^{\mu}_x[x^{\lambda-\nu};y] \\
&=\frac{B_y(\lambda+1,-\nu)}{\Gamma(-\nu)}\left(\frac{B_y(\lambda-\nu+1,-\mu)}{\Gamma(-\mu)}x^{\lambda-\mu-\nu}\right) \\
&=\frac{B_y(\lambda+1,-\nu)B_y(\lambda-\nu+1,-\mu)}{\Gamma(-\mu)\Gamma(-\nu)}x^{\lambda-\mu-\nu}; \\
\prescript{RL}{0}I^{\mu}_x\Big\{\prescript{RL}{0}I^{\nu}_x\{x^{\lambda};y\};y\Big\}&=\prescript{RL}{0}I^{\mu}_x\left\{\frac{B_{1-y}(\nu,\lambda+1)}{\Gamma(\nu)}x^{\lambda+\nu};y\right\} \\
&=\frac{B_{1-y}(\nu,\lambda+1)}{\Gamma(\nu)}\prescript{RL}{0}I^{\mu}_x\{x^{\lambda+\nu};y\} \\
&=\frac{B_{1-y}(\nu,\lambda+1)}{\Gamma(\nu)}\left(\frac{B_{1-y}(\mu,\lambda+\nu+1)}{\Gamma(\mu)}x^{\lambda+\mu+\nu}\right) \\
&=\frac{B_{1-y}(\nu,\lambda+1)B_{1-y}(\mu,\lambda+\nu+1)}{\Gamma(\mu)\Gamma(\nu)}x^{\lambda+\mu+\nu}; \\
\prescript{RL}{0}D^{\mu}_x\Big\{\prescript{RL}{0}D^{\nu}_x\{x^{\lambda};y\};y\Big\}&=\prescript{RL}{0}D^{\mu}_x\left\{\frac{B_{1-y}(-\nu,\lambda+1)}{\Gamma(-\nu)}x^{\lambda-\nu};y\right\} \\
&=\frac{B_{1-y}(-\nu,\lambda+1)}{\Gamma(-\nu)}\prescript{RL}{0}D^{\mu}_x\{x^{\lambda-\nu};y\} \\
&=\frac{B_{1-y}(-\nu,\lambda+1)}{\Gamma(-\nu)}\left(\frac{B_{1-y}(-\mu,\lambda-\nu+1)}{\Gamma(-\mu)}x^{\lambda-\mu-\nu}\right) \\
&=\frac{B_{1-y}(-\nu,\lambda+1)B_{1-y}(-\mu,\lambda-\nu+1)}{\Gamma(-\mu)\Gamma(-\nu)}x^{\lambda-\mu-\nu}.
\end{align*}
And there is no identity such as \[B_y(\lambda+1,\nu)B_y(\lambda+\nu+1,\mu)=B_y(\lambda+1,\mu+\nu)B(\mu,\nu)\] or \[B_{1-y}(\nu,\lambda+1)B_{1-y}(\mu,\lambda+\nu+1)=B_{1-y}(\mu+\nu,\lambda+1)B(\mu,\nu)\] for incomplete beta functions. Thus we surmise that there is no semigroup property for incomplete fractional differintegrals of either lower or upper type.

\begin{thm}
Let $b>0$, $0<y<1$, $x\in[0,b]$, and $f;[0,b]\rightarrow\mathbb{C}$.

If $f\in L^1[0,yb]$, then \[\lim_{\mu\rightarrow0^+}\prescript{RL}{0}I^{\mu}_x[f(x);y]=0,\] where $\mu\rightarrow0^+$ denotes convergence of $\mu$ towards $0$ within the right half plane $\mathrm{Re}(\mu)>0$.

If $f\in L^1[0,b]$ and $x$ is a Lebesgue point of $f$, then \[\lim_{\mu\rightarrow0^+}\prescript{RL}{0}I^{\mu}_x\{f(x);y\}=f(x),\] where $\mu\rightarrow0^+$ is as before.
\end{thm}

\begin{proof}
Firstly, we consider the lower incomplete RL integral. Here we are considering the quantity \[\frac{1}{\Gamma(-\mu)}\int_0^{yx}(x-t)^{-\mu-1}f(t)\,\mathrm{d}t,\quad\quad\mu\rightarrow0^+.\] The gamma reciprocal function $\frac{1}{\Gamma(z)}$ is entire with a zero at $z=0$, while the integrand is a well-behaved function of $t$ everywhere on the domain $[0,yx]$, so the limit is equal to zero as required. (The reason this argument does not work for the classical RL integral is due to the singularity at $t=x$, which is not included in the domain of the lower incomplete RL integral.)

For the upper incomplete RL integral, we need a more complicated argument. Recall the definition of Lebesgue points, namely that $x$ is a Lebesgue point of $f$ if \[\lim_{t\rightarrow0}\frac{1}{t}\int_0^t\left(f(x-u)-f(x)\right)\,\mathrm{d}u=0.\] Define \[F(t)=\int_{x-t}^xf(u)\,\mathrm{d}u=\int_0^tf(x-u)\,\mathrm{d}u,\] so that \[c(t)\coloneqq\frac{F(t)}{t}-f(x)=\frac{1}{t}\int_0^t\left(f(x-u)-f(x)\right)\,\mathrm{d}u\rightarrow0\text{ as }t\rightarrow0,\] by the Lebesgue property. Now, starting from \eqref{IRLdef:intupper1} and using integration by parts, we have
\begin{align*}
\prescript{RL}{0}I_x^{\mu}\{f(x);y\}&=\frac{1}{\Gamma(\mu)}\int_{yx}^x(x-t)^{\mu-1}f(t)\,\mathrm{d}t \\
&=\frac{1}{\Gamma(\mu)}\int_0^{(1-y)x}t^{\mu-1}f(x-t)\,\mathrm{d}t=\frac{1}{\Gamma(\mu)}\int_0^{(1-y)x}t^{\mu-1}F'(t)\,\mathrm{d}t \\
&=\frac{1}{\Gamma(\mu)}\left[t^{\mu-1}F(t)\right]_0^{(1-y)x}-\frac{\mu-1}{\Gamma(\mu)}\int_0^{(1-y)x}t^{\mu-2}F(t)\,\mathrm{d}t \\
&=\frac{x^{\mu-1}(1-y)^{\mu-1}F((1-y)x)}{\Gamma(\mu)}-\lim_{t\rightarrow0}\left[\frac{t^{\mu}}{\Gamma(\mu)}\frac{F(t)}{t}\right]-\frac{1}{\Gamma(\mu-1)}\int_0^{(1-y)x}t^{\mu-2}\left(tc(t)+tf(x)\right)\,\mathrm{d}t \\
&=\frac{x^{\mu-1}(1-y)^{\mu-1}F((1-y)x)}{\Gamma(\mu)}-\lim_{t\rightarrow0}\left[\frac{t^{\mu}}{\Gamma(\mu)}f(x)\right] \\
&\hspace{4cm}-\frac{1}{\Gamma(\mu-1)}\int_0^{(1-y)x}t^{\mu-1}c(t)\,\mathrm{d}t-\frac{f(x)}{\Gamma(\mu-1)}\int_0^{(1-y)x}t^{\mu-1}\,\mathrm{d}t \\
&=\frac{x^{\mu}(1-y)^{\mu}}{\Gamma(\mu)}\left(\frac{F((1-y)x)}{(1-y)x}-\frac{\mu-1}{\mu}f(x)\right)-\frac{1}{\Gamma(\mu-1)}\int_0^{(1-y)x}t^{\mu-1}c(t)\,\mathrm{d}t.
\end{align*}
Write $X=(1-y)x$, so that
\begin{equation}
\label{Lebesgue:eqn1}
\prescript{RL}{0}I_x^{\mu}\{f(x);y\}=\frac{X^{\mu}}{\Gamma(\mu)}\left(\frac{F(X)}{X}-\frac{\mu-1}{\mu}f(x)\right)-\frac{1}{\Gamma(\mu-1)}\int_0^{X}t^{\mu-1}c(t)\,\mathrm{d}t.
\end{equation}
We know that $c(t)\rightarrow0$ as $t\rightarrow0$, so for any $\epsilon>0$, there exists $\delta>0$ such that $0<t<\delta\Rightarrow|c(t)|<\epsilon$. We fix $\epsilon$ and argue from \eqref{Lebesgue:eqn1} as follows.
\begin{multline*}
\prescript{RL}{0}I_x^{\mu}\{f(x);y\}-f(x)=\frac{X^{\mu}}{\Gamma(\mu)}\frac{F(X)}{X}-\left(\frac{\mu-1}{\mu}\frac{X^{\mu}}{\Gamma(\mu)}+1\right)f(x) \\ -\frac{1}{\Gamma(\mu-1)}\int_0^{\delta}t^{\mu-1}c(t)\,\mathrm{d}t-\frac{1}{\Gamma(\mu-1)}\int_{\delta}^{X}t^{\mu-1}c(t)\,\mathrm{d}t.
\end{multline*}
As $\mu\rightarrow0^+$, we have:
\begin{align*}
\frac{X^{\mu}}{\Gamma(\mu)}\frac{F(X)}{X}&\rightarrow0; \\
\frac{\mu-1}{\mu}\frac{X^{\mu}}{\Gamma(\mu)}&=\frac{(\mu-1)X^{\mu}}{\Gamma(\mu+1)}\rightarrow-1; \\
\left|\frac{1}{\Gamma(\mu-1)}\int_0^{\delta}t^{\mu-1}c(t)\,\mathrm{d}t\right|&\leq\frac{\delta^{\mu}\epsilon}{\mu|\Gamma(\mu-1)|}=\frac{|\mu-1|\delta^{\mu}\epsilon}{\Gamma(\mu+1)}\rightarrow\epsilon; \\
\frac{1}{\Gamma(\mu-1)}\int_{\delta}^{X}t^{\mu-1}c(t)\,\mathrm{d}t&\rightarrow0;
\end{align*}
and therefore
\[\lim_{\mu\rightarrow0^+}\Big|\prescript{RL}{0}I_x^{\mu}\{f(x);y\}-f(x)\Big|\leq\epsilon.
\]
But since $\epsilon>0$ was arbitrary, this means the limit must in fact be $0$, which concludes the proof.
\end{proof}

\begin{lem}
\label{Leibniz:lemma}
Let $b>0$, $0<y<1$, $\mu\in\mathbb{C}$, and $n\in\mathbb{N}$. For any function $f:[0,b]\rightarrow\mathbb{C}$ in the appropriate function spaces given by Definitions \ref{IRL:lowerdefn}, \ref{IRL:upperdefn1}, or \ref{IRL:upperdefn2}, we have the following results:
\begin{align}
\label{Leibniz:lemma:lower} \prescript{RL}{0}D_x^{\mu}[x^nf(x);y]&=\sum_{k=0}^n\binom{n}{k}x^{n-k}(-1)^k\frac{\Gamma(-\mu+k)}{\Gamma(-\mu)}\prescript{RL}{0}D_x^{\mu-k}[f(x);y]; \\
\label{Leibniz:lemma:upper} \prescript{RL}{0}D_x^{\mu}\{x^nf(x);y\}&=\sum_{k=0}^n\binom{n}{k}x^{n-k}(-1)^k\frac{\Gamma(-\mu+k)}{\Gamma(-\mu)}\prescript{RL}{0}D_x^{\mu-k}\{f(x);y\}.
\end{align}
\end{lem}

\begin{proof}
The binomial theorem gives
\[t^n=(x-(x-t))^n=\sum_{k=0}^n\binom{n}{k}x^{n-k}(-1)^k(x-t)^k,\] so starting from the definition \eqref{IRLdef:intlower1} for $\mathrm{Re}(\mu)<0$, we have
\begin{align*}
\prescript{RL}{0}D_x^{\mu}[x^nf(x);y]&=\frac{1}{\Gamma(-\mu)}\int_0^{yx}(x-t)^{-\mu-1}f(t)\left[\sum_{k=0}^n\binom{n}{k}x^{n-k}(-1)^k(x-t)^k\right]\,\mathrm{d}t \\
&=\frac{1}{\Gamma(-\mu)}\sum_{k=0}^n\binom{n}{k}x^{n-k}(-1)^k\int_0^{yx}(x-t)^{-\mu+k-1}f(t)\,\mathrm{d}t \\
&=\sum_{k=0}^n\binom{n}{k}x^{n-k}(-1)^k\frac{\Gamma(-\mu+k)}{\Gamma(-\mu)}\prescript{RL}{0}D_x^{\mu-k}[f(x);y].
\end{align*}
This gives the result for lower incomplete fractional integrals ($\mathrm{Re}(\mu)<0$), which can easily be extended to all lower incomplete fractional differintegrals by analytic continuation. The proof for upper incomplete fractional differintegrals is exactly analogous.
\end{proof}

\begin{thm}[Incomplete fractional Leibniz rule]
\label{Leibniz}
Let $b>0$, $0<y<1$, $\mu\in\mathbb{C}$. For any function $f:[0,b]\rightarrow\mathbb{C}$ in the appropriate function spaces given by Definitions \ref{IRL:lowerdefn}, \ref{IRL:upperdefn1}, or \ref{IRL:upperdefn2}, and for any analytic function $g:[0,b]\rightarrow\mathbb{C}$, we have the following results:
\begin{align}
\label{Leibniz:lower} \prescript{RL}{0}D_x^{\mu}[f(x)g(x);y]&=\sum_{k=0}^{\infty}\binom{\mu}{k}\prescript{RL}{0}D_x^{\mu-k}[f(x);y]\prescript{RL}{0}D_x^kg(x); \\
\label{Leibniz:upper} \prescript{RL}{0}D_x^{\mu}\{f(x)g(x);y\}&=\sum_{k=0}^{\infty}\binom{\mu}{k}\prescript{RL}{0}D_x^{\mu-k}\{f(x);y\}\prescript{RL}{0}D_x^kg(x).
\end{align}
\end{thm}

\begin{proof}
Since $g$ is analytic, we can write \[g(t)=g(x-(x-t))=\sum_{k=0}^{\infty}\frac{(-1)^k}{k!}(x-t)^k\prescript{RL}{0}D_x^kg(x),\] where this series is locally uniformly convergent. Substituting this into the integral definition \eqref{IRLdef:intlower1} for the lower incomplete fractional integral ($\mathrm{Re}(\mu)<0$), we find:
\begin{align*}
\prescript{RL}{0}D_x^{\mu}[f(x)g(x);y]&=\frac{1}{\Gamma(-\mu)}\int_0^{yx}(x-t)^{-\mu-1}f(t)\left[\sum_{k=0}^{\infty}\frac{(-1)^k}{k!}(x-t)^k\prescript{RL}{0}D_x^kg(x)\right]\,\mathrm{d}t \\
&=\frac{1}{\Gamma(-\mu)}\sum_{k=0}^{\infty}\frac{(-1)^k}{k!}\prescript{RL}{0}D_x^kg(x)\int_0^{yx}(x-t)^{-\mu+k-1}f(t)\,\mathrm{d}t \\
&=\sum_{k=0}^{\infty}\frac{(-1)^k}{k!}\prescript{RL}{0}D_x^kg(x)\frac{\Gamma(-\mu+k)}{\Gamma(-\mu)}\prescript{RL}{0}D_x^{\mu-k}[f(x);y].
\end{align*}
(Note that we have used local uniform convergence of the Taylor series for $g$, in order to swap the order of summation and integration.) By the reflection formula for the gamma function, we have
\[\frac{\Gamma(-\mu+k)}{\Gamma(-\mu)}=\frac{\pi\sin(-\pi\mu)\Gamma(1+\mu)}{\pi\sin(\pi k-\pi\mu)\Gamma(1+\mu-k)}=\frac{(-1)^k\Gamma(1+\mu)}{\Gamma(1+\mu-k)},\]
which gives the desired result for lower incomplete fractional integrals. Once again, we can deduce the result for lower incomplete fractional derivatives by using analytic continuation, and then prove the result for upper incomplete fractional differintegrals in an entirely analogous fashion.
\end{proof}

\begin{thm}[Incomplete fractional chain rule]
\label{chain}
Let $b>0$, $0<y<1$, $\mu\in\mathbb{C}$. For any analytic composite function $f\circ g:[0,b]\rightarrow\mathbb{C}$, we have the following results:
\begin{align}
\label{chain:lower} \prescript{RL}{0}D_x^{\mu}[f(g(x));y]&=\sum_{k=0}^{\infty}\binom{\mu}{k}\frac{1-(1-y)^{k-\mu}}{\Gamma(1+k-\mu)}x^{k-\mu}\sum_{r=1}^k\frac{\mathrm{d}^rf(g(x))}{\mathrm{d}g(x)^r}\sum_{(r_1,\dots,r_k)}\Bigg[\prod_{j=1}^k\tfrac{j}{r_j!(j!)^{r_j}}\Big(\tfrac{\mathrm{d}^jg(x)}{\mathrm{d}x^j}\Big)^{r_j}\Bigg],
\\
\label{chain:upper} \prescript{RL}{0}D_x^{\mu}\{f(g(x));y\}&=\sum_{k=0}^{\infty}\binom{\mu}{k}\frac{(1-y)^{k-\mu}}{\Gamma(1+k-\mu)}x^{k-\mu}\sum_{r=1}^k\frac{\mathrm{d}^rf(g(x))}{\mathrm{d}g(x)^r}\sum_{(r_1,\dots,r_k)}\Bigg[\prod_{j=1}^k\tfrac{j}{r_j!(j!)^{r_j}}\Big(\tfrac{\mathrm{d}^jg(x)}{\mathrm{d}x^j}\Big)^{r_j}\Bigg],
\end{align}
where the innermost summation in each expression is taken over all $(r_1,\dots,r_k)\in\left(\mathbb{Z}^+_0\right)^m$ such that $\sum_jr_j=r$ and $\sum_jjr_j=k$.
\end{thm}

\begin{proof}
We apply Theorem \ref{Leibniz} to the product of the two functions $f\circ g(x)$ and $1$, where $f\circ g$ is analytic. This yields the following formulae:
\begin{align*}
\prescript{RL}{0}D_x^{\mu}[f(g(x));y]&=\sum_{k=0}^{\infty}\binom{\mu}{k}\prescript{RL}{0}D_x^{\mu-k}[1;y]\frac{\mathrm{d}^kf\circ g(x)}{\mathrm{d}x^k}; \\
\prescript{RL}{0}D_x^{\mu}\{f(g(x));y\}&=\sum_{k=0}^{\infty}\binom{\mu}{k}\prescript{RL}{0}D_x^{\mu-k}\{1;y\}\frac{\mathrm{d}^kf\circ g(x)}{\mathrm{d}x^k}.
\end{align*}
By Example \ref{power1}, we know that the incomplete fractional differintegrals of the constant function $1$ are given by:
\begin{align*}
\prescript{RL}{0}D_x^{\mu}[1;y]&=\frac{B_y(1,-\mu)}{\Gamma(-\mu)}x^{-\mu}=\frac{1-(1-y)^{-\mu}}{\Gamma(1-\mu)}x^{-\mu}; \\
\prescript{RL}{0}D_x^{\mu}\{1;y\}&=\frac{B_{1-y}(-\mu,1)}{\Gamma(-\mu)}x^{-\mu}=\frac{(1-y)^{-\mu}}{\Gamma(1-\mu)}x^{-\mu}.
\end{align*}
And by the classical Fa\`a di Bruno formula for repeated derivatives of a composite function, we also have
\begin{equation*}
\frac{\mathrm{d}^kf(g(x))}{\mathrm{d}x^k}=\sum_{r=1}^k\frac{\mathrm{d}^rf(g(x))}{\mathrm{d}g(x)^r}\sum_{(r_1,\dots,r_k)}\Bigg[\prod_{j=1}^k\tfrac{j}{r_j!(j!)^{r_j}}\Big(\tfrac{\mathrm{d}^jg(x)}{\mathrm{d}x^j}\Big)^{r_j}\Bigg],
\end{equation*}
where the inner summation is taken over all $(r_1,\dots,r_k)\in\left(\mathbb{Z}^+_0\right)^m$ such that $\sum_jr_j=r,\sum_jjr_j=k$.

Putting all of the above expressions together, we have the desired results.
\end{proof}

\section{Conclusions} \label{sec:concl}

In this work, we have performed a rigorous study and analysis of the recently defined incomplete fractional integrals of Riemann--Liouville type. Starting from the operators proposed in \cite{ozarslan-ustaoglu1}, we considered appropriate function spaces for their domain and range, and thence derived precise and rigorous definitions for these operators. We then considered how they interact with the standard differentiation operator, and deduced an extension of the definitions to incomplete fractional \textit{derivatives} as well as integrals.

Consideration of function spaces also yielded an unusual property of the lower incomplete fractional integral: acting on functions which are well-behaved on a small subinterval, it yields functions with larger domains of good behaviour. This extension property is a special feature of incomplete fractional calculus which may be useful in, for example, the theory of partial differential equations.

We also studied several important questions which are natural in any model of fractional calculus. Is a semigroup property satisfied? Are the fractional derivatives and integrals inverse to each other? How do they behave as the order of differintegration converges to zero? Is it possible to find fractional differintegrals for the product or composition of two functions? All of these questions are analysed and answered in the incomplete Riemann--Liouville fractional calculus, in order to flesh out the fundamentals of the theory.

%\section*{Acknowledgements}

\end{document}